\documentclass[11pt]{amsart}

\usepackage{amsmath}
\usepackage{amssymb}
\usepackage{amsthm}
\usepackage{graphics}
\usepackage{tikz}
\usepackage{tikz-qtree}
\usepackage{marvosym}
\usetikzlibrary{decorations.pathreplacing}
\usepackage{pgfplots}
\usepackage{amsfonts}
\usepackage{mathtools}
\usepackage{algorithm}
\usepackage{mathdots}
\usepackage[noend]{algpseudocode}
\usepackage{float}
\usepackage{caption}
\usepackage{subcaption}

\usepackage{color} 
\newcounter{examples}
\definecolor{ao(english)}{rgb}{0.0, 0.5, 0.0}
\newtheorem{theorem}{Theorem}
\newtheorem{lemma}[theorem]{Lemma}
\newtheorem{proposition}[theorem]{Proposition}
\newtheorem{corollary}[theorem]{Corollary}

{\theoremstyle{definition}
\newtheorem{example}[examples]{Example}}

\begin{document}
\title{Orphans in Forests of Linear Fractional Transformations}
\author{Sandie Han, Ariane M. Masuda, Satyanand Singh, \\ and Johann Thiel}
\date{\today}
\address{Department of Mathematics, New York City College of Technology (CUNY), 300 Jay Street,
Brooklyn, New York 11201}
\email{\{shan,amasuda,ssingh,jthiel\}@citytech.cuny.edu}


\begin{abstract}
The set of positive linear fractional transformations (PLFTs) is partitioned into an infinite forest of PLFT Calkin-Wilf-trees. The roots of these trees are called orphans.  In this paper, we provide a combinatorial formula for the number of orphan PLFTs with fixed determinant $D$. Then we provide a way of determining the orphan of a PLFT Calkin-Wilf-tree for a given PLFT.  In addition, we show that every positive complex number is the descendant of a complex $(u,v)$-orphan.
\end{abstract}
\maketitle


\section{Introduction}

In~\cite{CW}, Calkin and Wilf introduced a rooted infinite binary tree where every vertex is labeled by a positive rational number according to the following rules:
\begin{enumerate}
    \item[(CW1)] the root is labeled $1/1$,
    \item[(CW2)] the left child of a vertex $a/b$ is labeled $a/(a+b)$, and
    \item[(CW3)] the right child of a vertex $a/b$ is labeled $(a+b)/b$.
\end{enumerate}
Figure~\ref{fig:CWtree1} shows the first five rows of this tree, known as the Calkin-Wilf tree.

\begin{figure}[ht!]
\begin{center}
\begin{tikzpicture}[sibling distance=8pt]
\tikzset{level distance=30pt}
\Tree[.$\frac{1}{1}$ [.$\frac{1}{2}$ [.$\frac{1}{3}$ [.$\frac{1}{4}$ $\frac{1}{5}$ $\frac{5}{4}$ ] [. $\frac{4}{3}$ $\frac{4}{7}$ $\frac{7}{3}$ ] ]
   [.$\frac{3}{2}$ [.$\frac{3}{5}$ $\frac{3}{8}$ $\frac{8}{5}$ ] [.$\frac{5}{2}$ $\frac{5}{7}$ $\frac{7}{2}$ ] ] ] [.$\frac{2}{1}$ [.$\frac{2}{3}$ [.$\frac{2}{5}$ $\frac{2}{7}$ $\frac{7}{5}$ ] [.$\frac{5}{3}$ $\frac{5}{8}$ $\frac{8}{3}$ ] ] [.$\frac{3}{1}$ [.$\frac{3}{4}$ $\frac{3}{7}$ $\frac{7}{4}$ ] [.$\frac{4}{1}$ $\frac{4}{5}$ $\frac{5}{1}$ ] ] ]]
\end{tikzpicture}
\end{center}
\caption{The first five rows of the Calkin-Wilf tree.}\label{fig:CWtree1}
\end{figure}
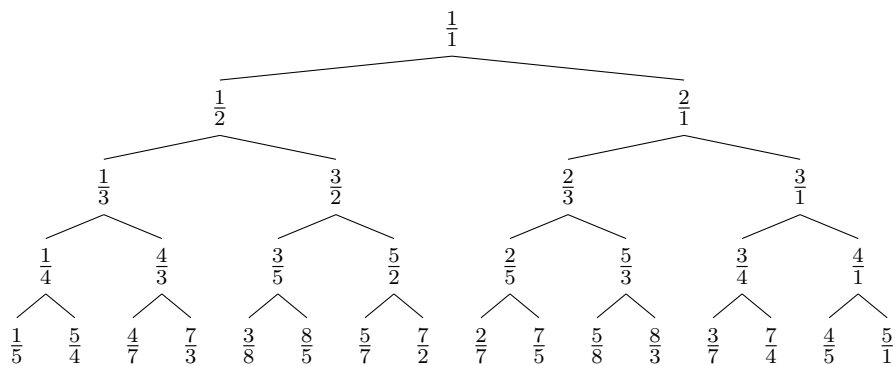

As noted by several authors~\cite{K,N1}, replacing $a/b$ in (CW2) and (CW3) above by the variable $z$ shows that the vertex labels of the Calkin-Wilf tree are generated by applying one of two transformations. For any vertex labeled $z$ in the Calkin-Wilf tree, the left child of $z$ is $L(z):=\frac{z}{z+1}$ and the right child of $z$ is $R(z):=z+1$. It is this observation that serves as the starting point of  a generalization of the Calkin-Wilf tree due to Nathanson~\cite{N1}.

By a positive linear fractional transformation (PLFT), we mean a function of the form
$$f(z) = \frac{az+b}{cz+d},$$ where $a, b, c$, and $d$ are nonnegative integers with $ad-bc\neq 0$. A special PLFT has the additional requirement that $ad-bc=1$. (Note that $L(z)$ and $R(z)$, the transformations used in connection to the Calkin-Wilf tree, are special PLFTs.)

Before moving forward, we mention some important facts regarding PLFTs that we will make use of repeatedly. Formal proofs of the following theorems can be found in ~\cite{N1}.

\begin{theorem}\label{lftmonoid}
The set of PLFTs forms a monoid under function composition. Furthermore, this monoid is isomorphic to $GL_2(\mathbb{N}_0)$ via the map $$\frac{az+b}{cz+d}\mapsto\begin{bmatrix}a & b\\ c & d\end{bmatrix}.$$
\end{theorem}

\begin{theorem}\label{lftmonoidfree}
The set of special PLFTs forms a free monoid of rank $2$, generated by $L(z)$ and $R(z)$, under function composition. Furthermore, the monoid is isomorphic to $SL_2(\mathbb{N}_0)$ via the map from Theorem~\ref{lftmonoid}.
\end{theorem}

Consider a rooted infinite binary tree  where every vertex is labeled
according to the following rules:
\begin{enumerate}
    \item[(P1)] the root is labeled by a PLFT $g(z)$,
    \item[(P2)] the left child of a vertex $f(z)$ is labeled $f(z)/(f(z)+1)$, and
    \item[(P3)] the right child of a vertex $f(z)$ is labeled $f(z)+1$.
\end{enumerate}
Note that Theorem~\ref{lftmonoid} ensures that the left child and right child of a PLFT $f(z)$ are also PLFTs. It quickly follows by induction that a tree generated using the above rules has all of its vertices labeled by a PLFT.

Such a tree will be referred to as a PLFT Calkin-Wilf tree (PLFT CW-tree) with root $g(z)$ and denoted by $\mathcal T(g(z))$. Figure~\ref{fig:LFTCWtree} shows the first four rows of $\mathcal T(z)$.

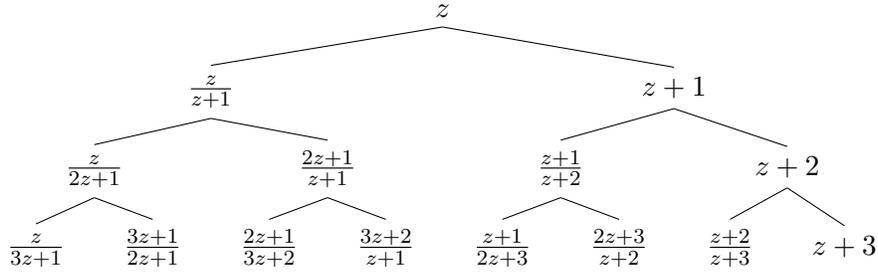
\begin{figure}[ht!]
\begin{center}
\begin{tikzpicture}[sibling distance=15pt]
\tikzset{level distance=30pt}
\Tree[.$z$ [.$\frac{z}{z+1}$ [.$\frac{z}{2z+1}$ $\frac{z}{3z+1}$ $\frac{3z+1}{2z+1}$ ]
   [.$\frac{2z+1}{z+1}$ $\frac{2z+1}{3z+2}$ $\frac{3z+2}{z+1}$ ] ] [.$z+1$ [.$\frac{z+1}{z+2}$ $\frac{z+1}{2z+3}$ $\frac{2z+3}{z+2}$ ]
   [.$z+2$ $\frac{z+2}{z+3}$ $z+3$ ] ]]
\end{tikzpicture}
\end{center}
\caption{The first four rows of $\mathcal T(z)$.}\label{fig:LFTCWtree}
\end{figure}

Theorem~\ref{lftmonoid} shows that we can associate a unique matrix in $GL_2(\mathbb{N}_0)$ with each PLFT in a natural way. Furthermore, the isomorphism between the two sets shows that we can compute the vertices of a PLFT CW-tree via matrix multiplication by the matrices $$L_1 :=\begin{bmatrix} 1 & 0\\ 1 & 1\end{bmatrix}\text{ and }R_1 :=\begin{bmatrix} 1 & 1\\ 0 & 1\end{bmatrix}.$$ Throughout the rest of this article, we will freely switch between either set, depending on the circumstances. As an example, Figure~\ref{fig:MLFTCWtree} shows the first four rows of the tree of matrices associated with $\mathcal T(z)$ (Figure~\ref{fig:LFTCWtree}).

\begin{figure}[ht!]
\begin{center}
\begin{tikzpicture}[every tree node/.style={font=\tiny,anchor=base}, sibling distance=10pt]
\tikzset{level distance=35pt}
\Tree[.${\begin{bmatrix} 1 & 0\\ 0 & 1\end{bmatrix}}$ [.${\begin{bmatrix} 1 & 0\\ 1 & 1\end{bmatrix}}$ [.${\begin{bmatrix} 1 & 0\\ 2 & 1\end{bmatrix}}$ ${\begin{bmatrix} 1 & 0\\ 3 & 1\end{bmatrix}}$ ${\begin{bmatrix} 3 & 1\\ 2 & 1\end{bmatrix}}$ ]
   [.${\begin{bmatrix} 2 & 1\\ 1 & 1\end{bmatrix}}$ ${\begin{bmatrix} 2 & 1\\ 3 & 2\end{bmatrix}}$ ${\begin{bmatrix} 3 & 2\\ 1 & 1\end{bmatrix}}$ ] ] [.${\begin{bmatrix} 1 & 1\\ 0 & 1\end{bmatrix}}$ [.${\begin{bmatrix} 1 & 1\\ 1 & 2\end{bmatrix}}$ ${\begin{bmatrix} 1 & 1\\ 2 & 3\end{bmatrix}}$ ${\begin{bmatrix} 2 & 3\\ 1 & 2\end{bmatrix}}$ ]
   [.${\begin{bmatrix} 1 & 2\\ 0 & 1\end{bmatrix}}$ ${\begin{bmatrix} 1 & 2\\ 1 & 3\end{bmatrix}}$ ${\begin{bmatrix} 1 & 3\\ 0 & 1\end{bmatrix}}$ ] ]]
\end{tikzpicture}
\end{center}
\caption{The first four rows of the matrix tree associated with $\mathcal T(z)$.}\label{fig:MLFTCWtree}
\end{figure}
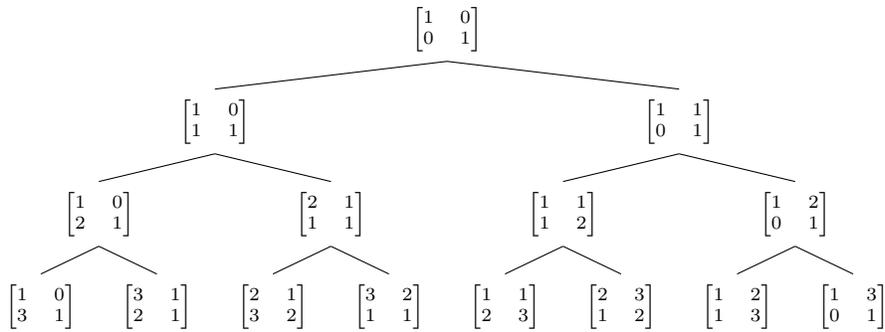

One remarkable property of the original Calkin-Wilf tree is that it produces an enumeration of the positive rationals~\cite{CW}. With the exception of the number 1 (the root), every positive rational number has a parent in this tree. While Theorem~\ref{lftmonoidfree} shows that a similar result holds for special PLFTs, this is not the case for the set of {\it all} PLFTs.

Luckly, not all is lost in this generalization. From~\cite{N1}, we find that the set of PLFTs is partitioned into an infinite forest of PLFT CW-trees. That is, each PLFT belongs to a unique such tree. The roots of these tress (which are not the children of any other PLFT) are called orphans and they are of the form $\frac{az+b}{cz+d}$ with either $a<c$ and $b>d$ or, alternatively, $a>c$ and $b<d$. The goal of this article is to further explore this set of orphans.



\section{The function $h(D)$}

As Nathanson~\cite[Theorem 7]{N3} showed, every PLFT CW-tree is rooted. In particular, every PLFT is the descendent of a unique orphan. Furthermore, while there are infinitely many such orphans, there are only finitely many with fixed determinant $D\neq 0$.

To this end, Nathanson~\cite{N3} defines the function $h(D)$ as the count of orphan PLFTs with determinant $D$ and computes the value of the function for $1\leq D\leq 15$ (see Figure~\ref{fig:hvals} and Figure~\ref{fig:hd}). (Note that $h(D)=h(-D)$, so we only consider positive values of $D$ from this point on.) Our goal in this section is to further explore some of the properties of $h(D)$.

\begin{figure}[ht!]
\centering
\begin{tabular}{|c||c|c|c|c|c|c|c|c|c|c|c|c|c|c|c|}
\hline
$D$ & 1 & 2 & 3 & 4 & 5 & 6 & 7 & 8 & 9 & 10 & 11 & 12 & 13 & 14 & 15\\
\hline
$h(D)$ & 1 & 4 & 7 & 13 & 15 & 26 & 25 & 39 & 40 & 54 & 49 & 79 & 63 & 88 & 88\\
\hline

\hline
\end{tabular}
  \caption{Values of $h(D)$ for $1\leq D\leq 15$.}
  \label{fig:hvals}
\end{figure}
We begin by showing that $h(D)$ is closely related to a partition function studied by Andrews~\cite{An}.

\begin{proposition}\label{hdvalue}
Let $\nu_2(D)$ denote the number of partitions of a positive integer $D$ using exactly two types of parts, $\sigma(D)$ denote the sum of divisors of $D$, and $\tau(D)$ denote the number of divisors of $D$. Then $$h(D)=\nu_2(D)+2\sigma(D)-\tau(D).$$
\end{proposition}

\begin{proof} From~\cite{N3}, we have that
\begin{align}
h(D) & = \sum_{\substack{b,c\geq0\\ b+c< D}} \sum_{\substack{a>c\\d>b\\ad=D+bc}}1.\label{nsum}
\end{align}

We split the double sum in~\eqref{nsum} into three cases: $b,c\geq 1$, $b=0$ and $c\geq 1$, and $b=c=0$. Notice that we need not consider the case $c=0$ and $b\geq 1$ separately, as the count is identical to the case $b=0$ and $c\geq 1$. So
\begin{align}
h(D) & = \sum_{\substack{b,c\geq1\\ b+c< D}} \sum_{\substack{a>c\\d>b\\ad=D+bc}}1 + 2\sum_{c=1}^{D-1}\sum_{\substack{a>c\\a\mid D}}1+\tau(D)\notag\\
& = \sum_{\substack{b,c\geq1\\ b+c< D}} \sum_{\substack{a>c\\d>b\\ad=D+bc}}1 + 2(\sigma(D)-\tau(D))+\tau(D)\notag\\
& = \sum_{\substack{b,c\geq1\\ b+c< D}} \sum_{\substack{a>c\\d>b\\ad=D+bc}}1 + 2\sigma(D)-\tau(D).\label{dsumf}
\end{align}
It remains to show that the double sum in~\eqref{dsumf} is equal to $\nu_2(D)$. To do this, notice that if $b,c\geq 1$ with $a>c$ and $d>b$, then $a=c+\epsilon_1$ and $d=b+\epsilon_2$, where $\epsilon_1,\epsilon_2>0$. So $ad-bc=a\cdot\epsilon_2+\epsilon_1\cdot b=D$. Since $a>\epsilon_1$, we have that each term in the sum corresponds to a partition of $D$ into exactly two types of parts (the parts being $a$ and $\epsilon_1$). Likewise, it is now easy to see how to turn a partition of $D$ using exactly two types of parts into a set of values $a,b,c,d$ that satisfy the requirements of the sum. See Figure \ref{fig:pd} and~\cite{An} for a geometric interpretation of this part of the sum.
\end{proof}

\begin{figure}[ht!]
\centering
\begin{tikzpicture}[scale=1.1]
    \coordinate (y) at (0,5);
    \coordinate (x) at (5,0);
    \fill[gray!30] (0,1.5) -- (2,1.5) -- (2,0) -- (0,0) -- cycle;
    \fill[gray!50] (0,1.5) -- (0,3) -- (3,3) -- (3,0) -- (2,0) -- (2,1.5) -- cycle;
    \draw[<->] (y) -- (0,0) --  (x);
    \draw[thick] (0,4) -- (4,0);
    \draw (-0.15,1.5) node [left] {$b$} -- (2,1.5) -- (2,-0.15) node [below] {$c$};
    \draw (-0.15,3) node [left] {$d$} -- (3,3) -- (3,-0.15) node [below] {$a$};
    \draw (-0.15,0) node [left] {$0$} -- (0,0);
    \draw (0,-0.15) node [below] {$0$} -- (0,0);
    \draw (-0.15,4) node [left] {$D-1$} -- (0,4);
    \draw (4,-0.15) node [below] {$D-1$} -- (4,0) ;
    \draw[dashed] (2,1.5) -- (3,1.5);
    \draw [decorate,decoration={brace,amplitude=5pt},xshift=-4pt,yshift=0pt] (-0.15,1.7) -- (-0.15,2.8) node [black,midway,xshift=-0.4cm] {$\epsilon_2$};
    \draw [decorate,decoration={brace,amplitude=5pt,mirror},xshift=0pt,yshift=-4pt] (2.2,-0.15) -- (2.8,-0.15) node [black,midway,yshift=-0.4cm] {$\epsilon_1$};
    \draw (2,2.5) to [out=90,in=180] (3,3.2)
to [out=0,in=240] (4,3.5) node [above] {$D$ (area)};
  \end{tikzpicture}
  \caption{Geometric representation of $\nu_2(D)$ term in $h(D)$.}
  \label{fig:pd}
\end{figure}
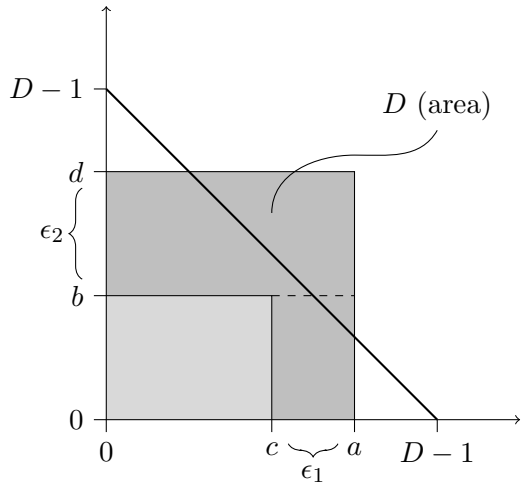

As a consequence of results of Ingham~\cite{In}, Estermann~\cite{Es}, and MacMahon~\cite{Ma}, we have the asymptotic behavior for $\nu_2(D)$, namely
\begin{align}
\nu_2(D) &\sim \frac{3}{\pi^2}\sigma(D)(\log{D})^2\label{v2asym}
\end{align}
as $D\to\infty$. From Proposition~\ref{hdvalue}, it follows that $h(D)$ has the same asymptotic behavior. Furthermore, from~\eqref{v2asym}, we can compute the summatory function of $h(D)$ in terms of a ``nicer" function that does not involve $\sigma(D)$. In particular, we get the following result (see Figure~\ref{fig:sumhd} and Figure~\ref{fig:ratiohd}).

Let $f(x)$ and $g(x)$ be functions. By $f(x)=O(g(x))$, we mean that there exists a constant $c$ such that $|f(x)|\leq c|g(x)|$ for all sufficiently large $x$.

\begin{proposition}\label{hdavg}
For large $x$, $$\sum_{D\leq x} h(D) = \frac{1}{4}x^2\log^2{x}+O(x^2\log{x}).$$
\end{proposition}

We give an independent proof of Proposition~\ref{hdavg} using elementary methods that do not require prior knowledge of~\eqref{v2asym}. Before we begin the proof of Proposition~\ref{hdavg}, we make note of a useful lemma.

\begin{lemma}\label{auxlemma}
For large $x$,
\begin{align}
\sum_{1\leq c\leq x-1}\sum_{c<a\leq x}\frac{1}{a(a-c)}=\frac{1}{2}\log^2{x}+O(\log{x}).\label{suplemma}
\end{align}
\end{lemma}
\begin{proof}

By partial fraction decomposition,
\begin{align*}
&\sum_{1\leq c\leq x-1}\sum_{c<a\leq x}\frac{1}{a(a-c)}  = \sum_{1\leq c\leq x-1}\sum_{c<a\leq x} \frac{1}{c(a-c)}-\frac{1}{ca}\\
 =& \sum_{1\leq c\leq x-1}\frac{1}{c}\Big(\log{(x-c)}+C+O\Big(\frac{1}{x-c}\Big)-\log{x}+\log{c}+O\Big(\frac{1}{c}\Big)\Big).\\
\end{align*}
The second line above follows from repeatedly applying the following well-known asymptotic formula for the harmonic series~\cite[Theorem 3.2]{Ap}
\begin{align}
\sum_{n\leq x}\frac{1}{n} &= \log{x}+C+O\Big(\frac{1}{x}\Big).\label{harm}
\end{align} (Note that $C$ is actually the Euler-–Mascheroni constant $\gamma$, however we will not need to know this for our particular application.) It follows that
\begin{align*}
\sum_{1\leq c\leq x-1}\sum_{c<a\leq x}\frac{1}{a(a-c)}& = \sum_{1\leq c\leq x-1}\frac{1}{c}\Big(\log{(x-c)}-\log{x}+\log{c}\Big)+O(\log{x})\\
& = \sum_{1\leq c\leq x-1}\Big(\frac{\log{c}}{c}+\frac{1}{c}\log{\Big(1-\frac{c}{x}\Big)}\Big)+O(\log{x}).
\end{align*}
Using the (alternating) Taylor series for $\log{(1-x)}$ for $|x|<1$, we get that $\big|\log{\big(1-\frac{c}{x}}\big)\big|<\frac{c}{x}.$ So
\begin{align*}
\sum_{1\leq c\leq x-1}\sum_{c<a\leq x}\frac{1}{a(a-c)} & =  \sum_{1\leq c\leq x-1}\frac{\log{c}}{c}+O(\log{x}).
\end{align*}
By partial summation
\begin{align*}
\sum_{1\leq c\leq x}\frac{\log{c}}{c} & = \frac{\log{x}}{x}(x+O(1))-\int_1^{x}(t+O(1))\Big(\frac{1}{t^2}-\frac{\log{t}}{t^2}\Big)\;dt\\
& = \int_1^{x}\frac{\log{t}}{t}\;dt+O(\log{x}) = \frac{1}{2}\log^2{x}+O(\log{x}),
\end{align*}
from which the desired result follows.
\end{proof}

\begin{proof}[Proof of Proposition~\ref{hdavg}] From~\eqref{dsumf}, it follows that
\begin{align}
\sum_{D\leq x} h(D) & = \sum_{D\leq x} \sum_{\substack{b,c\geq1\\ b+c< D}} \sum_{\substack{a>c\\d>b\\ad=D+bc}}1 + \sum_{D\leq x}(2\sigma(D)-\tau(D)).\label{startsum}
\end{align} Using~\cite[Theorem 3.3]{Ap} and~\cite[Theorem 3.4]{Ap}, we see that the contribution from rightmost sum in~\eqref{startsum} is $O(x^2\log{x})$.

Now let $$\Sigma = \sum_{D\leq x} \sum_{\substack{b,c\geq1\\ b+c< D}} \sum_{\substack{a>c\\d>b\\ad=D+bc}}1.$$ By rearranging the terms of the sum in $\Sigma$, we get that
\begin{align}
\Sigma & = \sum_{\substack{b\geq 1\\c\geq 1\\ b+c< x}} \sum_{\substack{a>c\\d>b\\ad\leq x+bc}}1\notag\\
& = \sum_{1\leq c\leq x-1}\sum_{1\leq b\leq x-1-c}\sum_{c<a\leq x}\sum_{b< d\leq \frac{x+bc}{a}}1\label{rearr1}\\
& = \sum_{2\leq c\leq x-1}\sum_{c<a\leq x}\sum_{0\leq b\leq x-1-c}\sum_{b< d\leq \frac{x+bc}{a}}1\label{rearr2}\\
& = \sum_{2\leq c\leq x-1}\sum_{c<a\leq x}\sum_{0\leq b\leq \frac{x-a}{a-c}}\sum_{b< d\leq \frac{x+bc}{a}}1.\label{rearr3}
\end{align}
Notice that the upper bound on the sum of $a$ in~\eqref{rearr1} can be restricted to values less than or equal to $x$ because otherwise, $ad>xd\geq x(1+b) = x+bx> x+bc$, a contradiction. We also have that~\eqref{rearr2} follows from the fact that the sums over $a$ and $b$ are independent of each other. Lastly,~\eqref{rearr3} follows from the fact that $b\leq \frac{x-a}{a-c}$ (otherwise we have a similar contradiction as above) and $\frac{x-a}{a-c}=\frac{x-c}{a-c}-1\leq x-1-c$. So
\begin{align*}
\Sigma & = \sum_{1\leq c\leq x-1}\sum_{c<a\leq x}\sum_{1\leq b\leq \frac{x-a}{a-c}}\Big(\frac{x+bc}{a}-b+O(1)\Big)\\
& = \sum_{1\leq c\leq x-1}\sum_{c<a\leq x}\sum_{1\leq b\leq \frac{x-a}{a-c}}\Big(\frac{x}{a}-\Big(\frac{c}{a}-1\Big)b+O(1)\Big)\\
& = \sum_{1\leq c\leq x-1}\sum_{c<a\leq x}\Big(\frac{x}{a}\Big(\frac{x-c}{a-c}+O(1)\Big)+\Big(\frac{c}{a}-1\Big)\Big(\frac{(x-a)^2}{2(a-c)^2}+O\Big(\frac{x-a}{a-c}\Big)\Big)\\
       & \qquad\qquad+O\Big(\frac{x-a}{a-c}\Big)\Big),
\end{align*}
where the last equality follows from the well-known formula for the sum of consecutive natural numbers. Now, using some basic algebraic manipulations and~\eqref{harm} once again,
\begin{align*}
\Sigma  &= \sum_{1\leq c\leq x-1}\sum_{c<a\leq x}\frac{x(x-c)}{a(a-c)}-\frac{(x-a)^2}{2a(a-c)}+O\Big(\frac{x}{a-c}\Big)\\
 &=  \sum_{1\leq c\leq x-1}\sum_{c<a\leq x}\frac{x^2}{2a(a-c)}-\frac{cx}{(a-c)}+O\Big(\frac{x}{a-c}\Big)\\
 &=  \frac{1}{2}x^2\sum_{1\leq c\leq x-1}\sum_{c<a\leq x}\frac{1}{a(a-c)}-x\sum_{1\leq c\leq x-1}\sum_{c<a\leq x}\frac{c}{(a-c)}+O(x^2\log{x})\\
 &=  \frac{1}{2}x^2\sum_{1\leq c\leq x-1}\sum_{c<a\leq x}\frac{1}{a(a-c)}+O(x^2\log{x}).
\end{align*}
The result then follows from Lemma~\ref{auxlemma}.

\begin{figure}
        \centering
        \begin{subfigure}[b]{0.5\textwidth}
                \includegraphics[width=\textwidth]{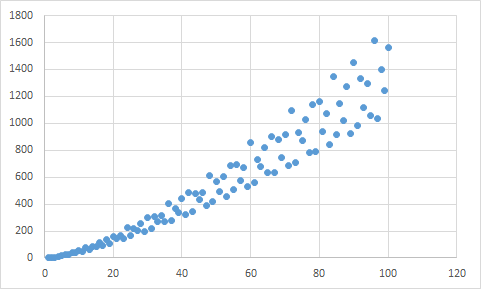}
                \caption{Plot of $h(D)$.}
                \label{fig:hd}
        \end{subfigure}%
            \qquad
        \begin{subfigure}[b]{0.5\textwidth}
                \includegraphics[width=\textwidth]{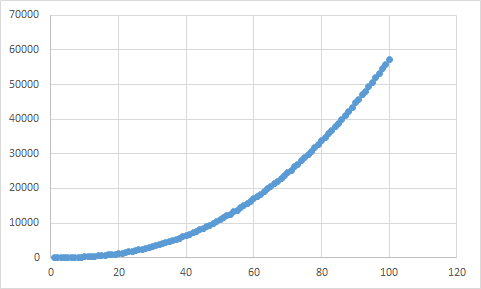}
                \caption{Plot of the summatory function of $h(D)$.}
                \label{fig:sumhd}
        \end{subfigure}
            \qquad
        \begin{subfigure}[b]{0.5\textwidth}
                \includegraphics[width=\textwidth]{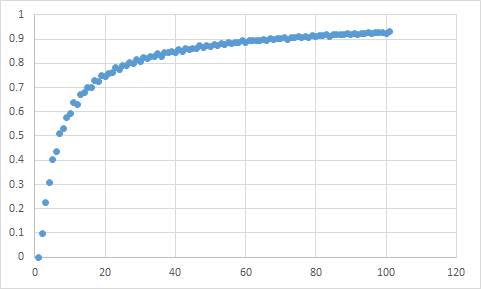}
                \caption{Ratio of the summatory function of $h(D)$ over $\dfrac{1}{4}x^2\log^2{x}$.}
                \label{fig:ratiohd}
        \end{subfigure}
        \caption{Plots related to $h(D)$.}\label{fig:hdplots}
\end{figure}
\end{proof}



\section{Positive linear fractional transformations and continued fractions}

Every positive rational number $\frac{a}{b}$ (usually written in lowest terms) can be expressed as
\[\dfrac{a}{b} = q_0+\dfrac{1}{q_1+\ddots +\dfrac{1}{q_{k-1}+\dfrac{1}{q_k}}}\]
where each $q_i\in\mathbb N_0$, $q_i>0$ for $i\ne 0$.
This continued fraction representation of $\frac{a}{b}$ is denoted by $[q_0,q_1,\ldots,q_k]$. Note that such a representation is not unique.

By using a procedure similar to the division algorithm for integers (see~\cite[Section 5]{N1} for an in-depth discussion), one can write any PLFT as \[\dfrac{az+b}{cz+d} = q_0+\dfrac{1}{q_1+\dfrac{1}{q_2+\ddots +\dfrac{1}{q_{k-1}+\dfrac{1}{q}}}}\] where each $q_i\in\mathbb N_0$, $q_i>0$ for $i\ne 0$, and $q:=q(z)$ is an orphan PLFT. We represent the above continued fraction of $\frac{az+b}{cz+d}$ by $[q_0,q_1,\ldots,q_{k-1},q]$. While $q$ is an orphan PLFT, it may not be orphan root of the PLFT CW-tree containing $\frac{az+b}{cz+d}$. In fact, either $q$ or $q^{-1}$ is the orphan root depending on the parity of $k$ ($q$ when $k$ is even and $q^{-1}$ otherwise).

\begin{example}\label{ex1}
Consider the PLFT $\frac{7z+8}{4z+5}$.  We have that
\[ \dfrac{7z+8}{4z+5} = 1+\dfrac{1}{1+\dfrac{1}{1+\dfrac{2z+1}{z+2}}}\\
  = \left [1,1,1,\dfrac{z+2}{2z+1}\right ].
\]
Here, $\frac{z+2}{2z+1}$ is an orphan, as it is not the left or right child of any PLFT. Furthermore, from Figure~\ref{fig:ex1tree}, we see that $\frac{2z+1}{z+2}$ is the root of the PLFT CW-tree containing $\frac{7z+8}{4z+5}$.
\end{example}

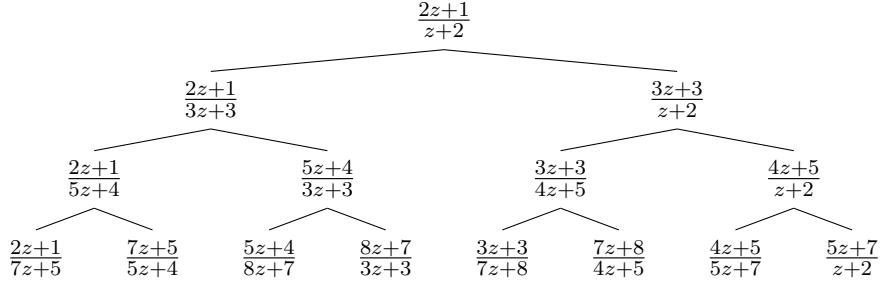
\begin{figure}[ht!]
\begin{center}
\begin{tikzpicture}[sibling distance=15pt]
\tikzset{level distance=30pt}
\Tree[.$\frac{2z+1}{z+2}$ [.$\frac{2z+1}{3z+3}$ [.$\frac{2z+1}{5z+4}$ $\frac{2z+1}{7z+5}$ $\frac{7z+5}{5z+4}$ ]
   [.$\frac{5z+4}{3z+3}$ $\frac{5z+4}{8z+7}$ $\frac{8z+7}{3z+3}$ ] ] [.$\frac{3z+3}{z+2}$ [.$\frac{3z+3}{4z+5}$ $\frac{3z+3}{7z+8}$ $\frac{7z+8}{4z+5}$ ]
   [.$\frac{4z+5}{z+2}$ $\frac{4z+5}{5z+7}$ $\frac{5z+7}{z+2}$ ] ]]
\end{tikzpicture}
\end{center}
\caption{The first four rows of $\mathcal T\left(\frac{2z+1}{z+2}\right)$.}\label{fig:ex1tree}
\end{figure}

The continued fractions of a positive rational number and its children in the Calkin-Wilf tree are closely related~\cite{HMST,N1}. A similar result holds for PLFT CW-trees.

\begin{lemma}\label{lrplft}
	Let $w$ be a PLFT with continued fraction representation $w=[q_0,q_1,\dots,q_r,q]$. Then $R(w)  = [q_0+1,q_1,\dots,q_r,q]$ and
	\begin{align*}
		L(w) &= \begin{cases}
		[0,q_1+1,\dots,q_r,q] & \text{ if }q_0=0,\\
		[0,1,q_0,q_1,\dots,q_r,q] & \text{ otherwise.}
		\end{cases}
	\end{align*}
\end{lemma}

In effect, Lemma~\ref{lrplft} shows that the continued fraction of a PLFT encodes its location relative to the root in its PLFT CW-tree. This result can be used to determine whether one PLFT is an ancestor of another within the same PLFT CW-tree.

\begin{example}\label{ex2}
Consider the PLFTs $\frac{7z+8}{4z+5}$, $\frac{3z+3}{4z+5}$, and $\frac{8z+7}{3z+3}$. We have that
\begin{align*}
\dfrac{7z+8}{4z+5} &= \left [1,1,1,\dfrac{z+2}{2z+1}\right ]=R\circ L\circ R\left(\frac{2z+1}{z+2}\right),\\
\dfrac{3z+3}{4z+5} &= \left [0,1,1,\dfrac{z+2}{2z+1}\right ]=L\circ R\left(\frac{2z+1}{z+2}\right),\text{ and}\\
\dfrac{8z+7}{3z+3} &= \left [2,1,\dfrac{2z+1}{z+2}\right ]=R\circ R\circ L\left(\frac{2z+1}{z+2}\right).
\end{align*}
We can clearly see that $\frac{3z+3}{4z+5}$ is an ancestor of $\frac{7z+8}{4z+5}$, but $\frac{8z+7}{3z+3}$ is not (see Figure~\ref{fig:ex1tree}). Using the original Calkin-Wilf tree, it is  easy to see the ancestor-descendant relations by noticing that $\frac{3}{4}$ and $\frac{3}{5}$ are ancestors of $\frac{7}{4}$ and $\frac{8}{5}$ respectively, but $\frac{8}{3}$ and $\frac{7}{3}$ are not (see Figure~\ref{fig:CWtree1}).
\end{example}

\begin{example}\label{ex3}
Consider the continued fraction representation of the PLFT $\frac{151z+119}{127z+100}$. A bit of work shows that
\[\frac{151z+119}{127z+100} = \bigg[1,5,3,1,\frac{3z+4}{4z+1}\bigg]\\
    = 1+\dfrac{1}{5+\dfrac{1}{3+\dfrac{1}{1+\dfrac{4z+1}{3z+4}}}}.\]
Furthermore,
\[	\frac{151}{127} = [1,5,3,2,3]
    = 1+\dfrac{1}{5+\dfrac{1}{3+\dfrac{1}{2+\dfrac{1}{3}}}}
    = 1+\dfrac{1}{5+\dfrac{1}{3+\dfrac{1}{1+\dfrac{4}{3}}}}
\]
and
\[ \frac{119}{100} = [1,5,3,1,4] = 1+\dfrac{1}{5+\dfrac{1}{3+\dfrac{1}{1+\dfrac{1}{4}}}}.
\]
\end{example}

Examples ~\ref{ex2} and ~\ref{ex3} suggest that there is a connection between the continued fractions of the rational numbers $\frac{a}{c}$ and $\frac{b}{d}$ (when $c,d\neq 0$) and the continued fraction of the PLFT $\frac{az+b}{cz+d}$. Our goal is to make this connection explicit while exploring some cases which are not as straight forward as Examples ~\ref{ex2} and ~\ref{ex3}.

Before stating some of our results, we want to establish the number of permissible zeros among the coefficients of a PLFT. In order that $ad-bc \neq 0$, there can be at most two zeros among the values $a$, $b$, $c$, and $d$. In the case where there are {\it exactly} two zeros, we have orphan PLFTs of the form $\frac{az}{d}$ or $\frac{b}{cz}$ with trivial continued fraction representations. Therefore, for the remainder of the section, we assume that at most one value among $a$, $b$, $c$, and $d$ is zero.

We begin with a useful lemma.

\begin{lemma}\label{gcdlemma}
 Let $w=\frac{az+b}{cz+d}$ be a PLFT and suppose that $L(w)=\frac{a'z+b'}{c'z+d'}$. Then $\gcd(a,c)=\gcd(a',c')$ and $\gcd(b,d)=\gcd(b',d')$. A similar result holds for $R(w)$.
\end{lemma}

\begin{proof}
From the definition of $L(\cdot)$, we see that
\begin{align*}
L(w) &= \frac{az+b}{(a+c)z+(b+d)},
\end{align*}
that is $a'=a$ and $c'=a+c$. We immediately get that $\gcd(a',c')=\gcd(a,a+c)=\gcd(a,c)$, as desired. The remaining portion of the lemma can be handled in a similar fashion.
\end{proof}

We now state the main theorem in this section. While it is not the most general statement that can be made, it is versatile enough to handle any case with some slight modifications.

\begin{theorem}\label{equivtheorem}
	Let $w=\frac{az+b}{cz+d}$ be a PLFT with $c,d\neq 0$. Then the following are equivalent\footnote{Note that the following representations appearing in the theorem are not necessarily the continued fractions of either the PLFTs or rational numbers.}:
    \begin{itemize}
    \item[(a)] We have that
    \begin{align*}
		w &= q_0+\dfrac{1}{q_1+\ddots+\dfrac{1}{q_{k-1}+\dfrac{a'z+b'}{c'z+d'}}}
	\end{align*}
    with $c'\neq 0$.
	\item[(b)] We have that
    \begin{align*}
		\frac{a}{c} &= q_0+\dfrac{1}{q_1+\ddots+\dfrac{1}{q_{k-1}+\dfrac{a''}{c''}}}
	\end{align*}
	with $c''\neq 0$ and $\gcd(a'',c'')=1$, and either
	\begin{align*}
		\frac{b}{d} &= q_0+\dfrac{1}{q_1+\ddots+\dfrac{1}{q_{k-1}+\dfrac{b''}{d''}}}
    \end{align*}
    with $d''\neq 0$ and $\gcd(b'',d'')=1$, or
    \begin{align*}
        \frac{b}{d} &= q_0+\dfrac{1}{q_1+\ddots+\dfrac{1}{q_{k-2}}}.
    \end{align*}
    \item[(c)] We have that
    \begin{align*}
    \begin{bmatrix}
    a & b \\
    c & d
    \end{bmatrix} &= \begin{cases}
    R_1^{q_0}L_1^{q_1}\cdots R_1^{q_{k-1}} \begin{bmatrix}
    a' & b' \\
    c' & d'
    \end{bmatrix} & \text{ when $k$ is odd},\\[3ex]
    R_1^{q_0}L_1^{q_1}\cdots L_1^{q_{k-1}} \begin{bmatrix}
    a' & b' \\
    c' & d'
    \end{bmatrix} & \text{ otherwise}.
    \end{cases}
    \end{align*}
    \end{itemize}
\end{theorem}

\begin{proof}
(a)$\Longrightarrow$(b): We obtain the first part of (b) by noting that
\begin{align*}
\frac{a}{c} &= \lim_{z\to\infty}w = q_0+\dfrac{1}{q_1+\ddots+\dfrac{1}{q_{k-1}+\displaystyle\lim_{z\to\infty} \dfrac{a'z+b'}{c'z+d'}}}\\
&= q_0+\dfrac{1}{q_1+\ddots+\dfrac{1}{q_{k-1}+\dfrac{a'}{c'}}},
\end{align*}
which gives the desired result with $a''=\frac{a'}{\gcd(a,c)}$ and $c''=\frac{c'}{\gcd(a,c)}$.

The second half of (b) follows similarly by taking the limit as $z\to 0^+$ of $w$. When $d'\neq 0$, we obtain, again, the desired result with $b''=\frac{b'}{\gcd(b,d)}$ and $d''=\frac{d'}{\gcd(b,d)}$. If $d'=0$, then
\begin{align*}
\frac{b}{d} &= \lim_{z\to 0^+}w = q_0+\dfrac{1}{q_1+\ddots+\displaystyle\lim_{z\to 0^+}\Bigg(\dfrac{1}{q_{k-1}+ \frac{a'z+b'}{c'z}}\Bigg)}\\
&= q_0+\dfrac{1}{q_1+\ddots+\dfrac{1}{q_{k-2}}},
\end{align*}
as desired.

(b)$\Longrightarrow$(a): Before we begin this portion of the proof, we introduce a bit of notation. For any PLFT $f(z)$, let $f^m(z)=f^{m-1}\circ f(z)$ for an integer $m>0$ and $f^0(z)=f(z)$. We will make use of this notation in the case where $f(z)$ is $L(z)$ or $R(z)$.

Suppose that $k$ is odd. Let $\frac{ez+f}{gz+h}$ be the PLFT given by
\begin{align*}
\dfrac{ez+f}{gz+h} &= R^{q_0}\circ L^{q_1}\circ\cdots R^{q_{k-1}}\bigg(\dfrac{a'z+b'}{c'z+d'}\bigg).
\end{align*}
where $a'=\gcd(a,c)\cdot a''$, $c'=\gcd(a,c)\cdot c''$, $b'=\gcd(b,d)\cdot b''$, and $d'=\gcd(b,d)\cdot d''$.
By Lemma~\ref{lrplft},
\begin{align}
\dfrac{ez+f}{gz+h} &= q_0+\dfrac{1}{q_1+\ddots+\dfrac{1}{q_{k-1}+\dfrac{a'z+b'}{c'z+d'}}}.\label{expansion}
\end{align}
Note that $g\neq 0$ since $c''\neq 0$. Taking the limit of both sides of~\eqref{expansion} as $z\to\infty$ shows that $\frac{e}{g}=\frac{a}{c}$. By repeatedly applying Lemma~\ref{gcdlemma}, it follows that $\gcd(e,g)=\gcd(a',c')=\gcd(a,c)$. This immediately gives that $e=a$ and $g=c$.

In the case when $d''\neq 0$, we get that $f=b$ and $h=d$ by taking the limit of both sides of~\eqref{expansion} as $z\to 0^+$ and repeating the above argument. When $d''=0$, the situation requires some extra computations.

If $k = 1$, then $w = q_0 + \frac{a'z+b'}{c'z+d'}$, which means that, in the case where $d''=0$, it follows that $d=0$. This contradicts our initial assumption about $d$, so we must have that $k>1$. Furthermore,
\begin{align*}
L^{q_{k-2}}\circ R^{q_{k-1}}\bigg(\dfrac{a'z+b'}{c'z}\bigg) &= \dfrac{1}{q_{k-2}+\dfrac{1}{q_{k-1}+\dfrac{a'z+b'}{c'z}}}\\
&= \dfrac{(a'+c'q_{k-1})z+b'}{(c'+a'q_{k-2}+c'q_{k-1}q_{k-2})z+b'q_{k-2}}.
\end{align*}
Using the above computation, we reduce the problem to the previous case with $k-2$ (which is nonnegative) replacing $k$ and $b'q_{k-2}$ (which is not 0) replacing $d'$. Taking limits as $z\to 0^+$, as before, we get that $\frac{e}{f}=\frac{b}{d}$ and $\gcd(e,f)=\gcd(b',b'q_{k-2})=b'=\gcd(b',0)=\gcd(b',d')=\gcd(b,d)$, where the first and last equalities are given by Lemma~\ref{gcdlemma}.

A similar argument works for the case where $k$ is even. Simply apply the above argument using $k-1$ (which is odd) on $L^{q_{k-1}}\left(\frac{a'z+b'}{c'z+d'}\right)$ instead of $\frac{a'z+b'}{c'z+d'}$. This completes this portion of the proof.

(a)$\Longleftrightarrow$(c): This equivalence follows from Lemma~\ref{lrplft} and Theorem~\ref{lftmonoid}.
\end{proof}

The following example shows that the PLFT $\frac{a'z+b'}{c'z+d'}$ in part (a) of Theorem~\ref{equivtheorem} is not unique or necessarily the orphan root associated with $\frac{az+b}{cz+d}$.

\begin{example}\label{ex4}
We have that
\[
  \dfrac{43}{30}=\left [1,2,3,4\right ]=1+\dfrac{1}{2+\dfrac{1}{3+\dfrac{1}{4}}}
\]
and
\[
  \dfrac{10}{7}=\left [1,2,3\right ]=1+\dfrac{1}{2+\dfrac{1}{3}}=1+\dfrac{1}{2+\dfrac{1}{3+\dfrac{0}{1}}}.
\]
Taking $k=3$ in Theorem~\ref{equivtheorem} part (b), and noting that $\gcd(43,30)=\gcd(10,7)=1$, it follows that
\[
  \dfrac{43z+10}{30z+7}=1+\dfrac{1}{2+\dfrac{1}{3+\dfrac{z}{4z+1}}}.
\]
however $\dfrac{z}{4z+1}$ is not a PLFT orphan.
Alternatively, taking $k=4$ in Theorem~\ref{equivtheorem} part (b), it follows that
\[
  \dfrac{43z+10}{30z+7}=1+\dfrac{1}{2+\dfrac{1}{3+\dfrac{1}{4+\dfrac{1}{z}}}}=\left [1,2,3,4,z\right ].
\]
where $z$ is a PLFT orphan and the orphan root associated with $\dfrac{43z+10}{30z+7}$.
\end{example}

Theorem~\ref{equivtheorem} assumes that $c$ and $d$ are nonzero. If this is not the case, then we apply the theorem to the PLFT $\frac{cz+d}{az+b}$ instead.

\begin{example}\label{ex5}
We have that
\[
  \dfrac{5}{7}=0+\dfrac{1}{1+\dfrac{1}{2+\dfrac{1}{2}}}=\left [0,1,2,2\right ].
\]
So
\[
  \dfrac{7z+1}{5z} = \dfrac{1}{\dfrac{5z}{7z+1}} = \dfrac{1}{0+\dfrac{1}{1+\dfrac{2z+1}{5z}}}\\
  = \left [1,\dfrac{2z+1}{5z}\right ],
\]
where we have applied Theorem~\ref{equivtheorem} to the PLFT $\dfrac{5z}{7z+1}$.
\end{example}

Theorem~\ref{equivtheorem} also assumes that the continued fraction of $\frac{a}{c}$ is ``longer" than that of $\frac{b}{d}$. If this is not the case, then we apply the theorem to the PLFT $\frac{bz+a}{dz+c}$ instead and recover the original PLFT by the change of variables $z\mapsto \frac{1}{z}$.

\begin{example}\label{ex6}
Using Example~\ref{ex4}, we see that, by letting $y=1/z$,
\[
\dfrac{10z+43}{7z+30} = \dfrac{43y+10}{30y+7} = 1+\dfrac{1}{2+\dfrac{1}{3+\dfrac{1}{4+\dfrac{1}{y}}}} =\left [1,2,3,4,\frac{1}{z}\right ].
\]
\end{example}

All of the examples given so far have been selected with $\gcd(a,c)=\gcd(b,d)=1$. This need not always be the case. Given two {\it distinct} PLFTs $\frac{az+b}{cz+d}$ and $\frac{ez+f}{gz+h}$ with $\frac{a}{c}=\frac{e}{g}$ and $\frac{b}{d}=\frac{f}{h}$, we expect their continued fractions to be different even though the continued fractions of $\frac{a}{c}$ and $\frac{e}{g}$, as well as those of $\frac{b}{d}$ and $\frac{f}{h}$, are identical. Lemma~\ref{gcdlemma} accounts for this potential difference and shows that the only modification needed for the non-relatively prime case is to adjust the values of the relatively prime case in a simple way.

\begin{example}\label{ex7}
We have that
\[
  \dfrac{86}{60}=\dfrac{43}{30}=\left [1,2,3,4\right ]
\qquad
\text{and}
\qquad
  \dfrac{30}{21}=\dfrac{10}{7}=\left [1,2,3\right ].
\]
Taking $k=4$ in Theorem~\ref{equivtheorem} part (b), and noting that $\gcd(86,60)=2$ and $\gcd(30,21)=3$, it follows that
\[
  \dfrac{86z+30}{60z+21}=1+\dfrac{1}{2+\dfrac{1}{3+\dfrac{1}{4+\dfrac{3}{2z}}}}=\left [1,2,3,4,\dfrac{2z}{3}\right ].
\]
\end{example}

In~\cite{HMST}, we found explicit conditions for a rational number to be the descendant of another rational number in the Calkin-Wilf tree based on their continued fractions. We describe the conditions below and provide the continued fractions of the ancestors of a rational number. We will make use of Proposition~\ref{ADR} (part (c) in particular) in Theorem~\ref{orphanroot} when selecting ancestors of given rational numbers.

\begin{proposition}[Descendant Conditions]\label{ADR}
Suppose that $w$ and $w'$ are positive rational numbers with continued fraction representations
$w=[q_0,q_1,\ldots,$ $q_r]$ and $w'=[p_0,p_1,\ldots,p_s]$. Then the following statements are equivalent:
\begin{enumerate}
\item[(a)] $w'$ is a descendant of $w$ in the Calkin-Wilf tree;
\item[(b)]  $s\ge r$, $2\mid (s-r)$,  $p_{s-r+i}=q_i$ for $2\le i \le r$, and
\[\begin{cases}
p_{s-r}\ge q_0 \text{ and } p_{s-r+1}=q_1 & \text{ if } q_0\ne 0,\\
p_{s-r+1}\ge q_1 & \text{ otherwise};
\end{cases}\]
\item[(c)]  $w^{(-1)^j}=[k,p_{j+1},\ldots,p_s]$ for $j,k\in\mathbb{N}$, $0\leq k < p_j$, $0 \leq j \leq s-1$.
\end{enumerate}
\end{proposition}

In order to obtain the orphan root of a PLFT using Theorem~\ref{equivtheorem}, the values of $k$ and $q_{k-1}$ must be maximized. This is done by first selecting $k$ as large as possible and then (with $k$ fixed) selecting $q_{k-1}$ as large as possible. In some cases, the largest value of $k$ is obtained by considering alternative forms of the continued fraction representations of $\frac{a}{c}$ and $\frac{b}{d}$. Since we are maximizing over a finite set of choices, we can always attain the maximum and find the orphan root. Any pair of representations that allows for such a maximization will be referred to as an optimal pair. We summarize the above discussion explicitly in Theorem~\ref{orphanroot} below.

\begin{theorem}\label{orphanroot}
Let $w=\frac{az+b}{cz+d}$ be a PLFT with $c,d\neq 0$ and suppose that $\frac{a}{c} = [q_0,q_1,\ldots,q_r]$ and $\frac{b}{d} = [q'_0,q'_1,\ldots,q'_s]$ form an optimal pair of continued fraction representations. Furthermore, assume that $2\leq s\leq r$ and that $q_s>q'_s$ if $r=s$.  Let $k$ be the largest integer such that $2\leq k\leq s+2$, and $q_i=q'_i$, for $i=0,1,\ldots,k-2$. Then there exists a positive integer $p$ such that
$$\dfrac{az+b}{cz+d}=\left[q_0,q_1,\ldots,q_{k-2},p,\left(\dfrac{a'z+b'}{c'z+d'}\right)^{-1}\right]$$
where $\frac{a'z+b'}{c'z+d'}$ is a PLFT orphan.

 \begin{itemize}
    \item[(a)] If $k\leq s+1$, then $p=\min(q_{k-1},q'_{k-1})$, and the orphan root of $w$ is $\left(\frac{a'z+b'}{c'z+d'}\right)^{(-1)^{k-1}}$
where $\frac{a'}{c'}=[q_{k-1}-p,q_k,\dots,q_r]$ with $\gcd(a',c')=\gcd(a,c)$, and
$\frac{b'}{d'}=[q'_{k-1}-p,q'_k,\dots,q'_s]$ with $\gcd(b',d')=\gcd(b,d)$ if $k<s+1$ or $b'=0$ and $d'=\gcd(b,d)$ if $k=s+1$.

  \item[(b)] If  $k=s+2$, then $p=q_{s+1}$ and the orphan root of $w$ is $\left(\frac{a'z+\gcd(b,d)}{c'z}\right)^{(-1)^{s+1}}$ where $\frac{a'}{c'}=[q_{s+2},\ldots,q_r]$.
\end{itemize}

\end{theorem}

We present two proofs of Theorem~\ref{orphanroot}. The first proof below makes use of Theorem~\ref{equivtheorem}. The second proof, appearing after Corollary~\ref{rootzmatrix}, establishes the same result from a matrix perspective.

\begin{proof}[First Proof of Theorem~\ref{orphanroot}]
Suppose that $k\leq s+1$. Then
\[
	\frac{a}{c} = q_0+\dfrac{1}{q_1+\ddots+\dfrac{1}{q_{r-1}+\dfrac{1}{q_r}}}\\
    = q_0+\dfrac{1}{q_1+\ddots+\dfrac{1}{q_{k-2}+\dfrac{1}{p+\dfrac{a''}{c''}}}}
\]
and, if $k<s$,
\[
    \frac{b}{d} = q'_0+\dfrac{1}{q'_1+\ddots+\dfrac{1}{q'_{s-1}+\dfrac{1}{q'_s}}}\\
    = q_0+\dfrac{1}{q_1+\ddots+\dfrac{1}{q_{k-2}+\dfrac{1}{p+\dfrac{b''}{d''}}}}
\]
where $\frac{a''}{c''}=[q_{k-1}-p,q_k,\ldots,q_r]$ with $\gcd(a'',c'')=1$, and
$\frac{b''}{d''}=[q'_{k-1}-p,q'_k,\dots,q'_s]$ with $\gcd(b',d')=1$. Using Lemma~\ref{gcdlemma} and Theorem~\ref{equivtheorem} (b), this implies that
\begin{align*}
    w &= q_0+\dfrac{1}{q_1+\ddots+\dfrac{1}{q_{k-2}+\dfrac{1}{p+\dfrac{a'z+b'}{c'z+d'}}}}.
\end{align*}
By the definition of $p$, out of the two fractions $\frac{a'}{c'}$ and $\frac{d'}{b'}$, one must be greater than 1 and one must be smaller than 1. So $\left(\frac{c'z+d'}{a'z+b'}\right)^{(-1)^{k-1}}$ is the orphan root of $w$. This gives the desired continued fraction representation of $w$ when $k<s+1$. When $k=s+1$, the above argument works with $b''=0$.

When $k=s+2$, we see that
\[
	\frac{a}{c} = q_0+\dfrac{1}{q_1+\ddots+\dfrac{1}{q_{r-1}+\dfrac{1}{q_r}}}\\
    = q_0+\dfrac{1}{q_1+\ddots+\dfrac{1}{q_{k-2}+\dfrac{1}{q_{k-1}+\dfrac{a''}{c''}}}}
\]
and
\[
    \frac{b}{d} = q'_0+\dfrac{1}{q'_1+\ddots+\dfrac{1}{q'_{s-1}+\dfrac{1}{q'_s}}}\\
    = q_0+\dfrac{1}{q_1+\ddots+\dfrac{1}{q_{k-2}}}
\]
Again, by Lemma~\ref{gcdlemma} and Theorem~\ref{equivtheorem} (b) (in the case where the continued fraction of $\frac{b}{d}$ is ``shorter"), this implies that
\begin{align*}
    w &= q_0+\dfrac{1}{q_1+\ddots+\dfrac{1}{q_{k-2}+\dfrac{1}{q_{k-1}+\dfrac{a'z+\gcd(b,d)}{c'z}}}},
\end{align*}
as desired.
\end{proof}

\begin{example}\label{ex8}
We have that
\[
    \dfrac{27}{19}=[1,2,2,1,2]
\qquad
\text{and}
\qquad
    \dfrac{10}{7}=[1,2,3].
\]
Taking $k=3$ and $q_2=2$ in Theorem~\ref{equivtheorem} part (b), it follows that
\[
    \dfrac{27z+10}{19z+7}=1+\dfrac{1}{2+\dfrac{1}{2+\dfrac{2z+1}{3z+1}}}.
\]
Despite the fact that we have taken $k$ and $q_{k-1}$ to be as large as possible given the above continued fractions of $\frac{27}{19}$ and $\frac{10}{7}$, the PLFT $\frac{2z+1}{3z+1}$ is {\bf not} an orphan. However, if we consider alternatively
\[
    \dfrac{10}{7}=[1,2,2,1],
\]
taking $k=5$, we obtain the orphan root $\frac{1}{z}$, and the continued fraction
\[
    \dfrac{27z+10}{19z+7}=\left[1,2,2,1,2,z\right ].
\]
\end{example}

Theorem~\ref{orphanroot} implies the following result in the case where $w$ is a PLFT with $ad-bc=\pm 1$.

\begin{corollary}\label{rootz}
Let $w=\frac{az+b}{cz+d}$ be a PLFT with $c,d\neq 0$ and $ad-bc=\pm 1$ such that $\frac{a}{c}  = [q_0,q_1,\ldots,q_{s+1}]$ and $\frac{b}{d} = [q_0,q_1,\ldots,q_s]$. Then
$$\dfrac{az+b}{cz+d}=[q_0,q_1,\ldots,q_{s+1},z]$$
whose orphan root is $z^{(-1)^s}$.

\end{corollary}

\begin{proof}
The corollary follows immediately from case (b) in Theorem~\ref{orphanroot}.
\end{proof}

By translating Corollary~\ref{rootz} into the setting for matrices (using part (c) of Theorem~\ref{equivtheorem}), we get the following result.
\begin{corollary}\label{rootzmatrix}
The matrix $M=\begin{bmatrix} a & b \\
c & d
\end{bmatrix}$ is in the monoid generated by $L_1$ and $R_1$  if and only if $M\in\{I_2,R_1,L_1,R_1^2,L_1^2,\dots\}$ or $\frac{a}{c}=[q_0,q_1,\ldots,q_{s+1}]$ and $\frac{b}{d}=[q_0,q_1,\ldots,q_s]$ with $ad-bc=\pm 1$. Furthermore, in the latter case,
\begin{align*}
M =  \begin{cases}
    R_1^{q_0}L_1^{q_1}\cdots R_1^{q_r}\text{ when $r$ is odd},\\
    R_1^{q_0}L_1^{q_1}\cdots L_1^{q_r}\text{ otherwise}.
    \end{cases}
\end{align*}
\end{corollary}
In other words, the continued fractions of $\frac{a}{c}$ and $\frac{b}{d}$ (when appropriate) encode the decomposition of a matrix $M=\begin{bmatrix} a & b \\
c & d
\end{bmatrix}$ with determinant 1 as a product of positive powers of $L_1$ and $R_1$. This is a known result (see~\cite[Section 2]{B}).

Before we give the second proof of Theorem~\ref{orphanroot}, we need to clarify the notation in the proof. Theorem~\ref{lftmonoid} shows that we can associate any PLFT $\frac{az+b}{cz+d}$ with the matrix $\begin{bmatrix} a & b\\c & d \end{bmatrix}$ and we can compute its descendants via matrix multiplication. A similar idea can be applied to rational numbers. Instead of associating a matrix to the rational number $\frac{a}{b}$, we can associate it with the vector $\begin{bmatrix} a\\b \end{bmatrix}$ (see~\cite{HMST} for details).

\begin{proof}[Second Proof of Theorem~\ref{orphanroot}] Suppose $k\leq s+1$. If $k$ is odd,
$\frac{a}{c}$ is associated to $$R_1^{q_0}L_1^{q_1}\cdots L_1^{q_{k-2}} R_1^{p} \begin{bmatrix} a'\\c' \end{bmatrix}$$ where $\frac{a'}{c'}=[q_{k-1}-p,q_k,\ldots,q_r]$ with $\gcd(a',c')=\gcd(a,c)$, and
 $\frac{b}{d}$ is associated to $$R_1^{q_0}L_1^{q_1}\cdots L_1^{q_{k-2}}R_1^{p} \begin{bmatrix} b'\\d' \end{bmatrix}$$ where $\frac{b'}{d'}=[q'_{k-1}-p,q'_k,\dots,q'_s]$ with $\gcd(b',d')=\gcd(b,d)$ if $k<s+1$ or $b'=0$ and $d'=\gcd(b,d)$ if $k=s+1$. This implies that $R_1^{q_0}L_1^{q_1}\cdots L_1^{q_{k-2}} R_1^{p} \begin{bmatrix} a' & b'\\c' & d'\end{bmatrix}$. By Theorem~\ref{equivtheorem}, it follows that
$$\dfrac{az+b}{cz+d}=\left[q_0,q_1,\ldots,q_{k-2},p,\left(\dfrac{a'z+b'}{c'z+d'}\right)^{-1}\right].$$ The proof is similar in the case where $k$ is even.

Suppose that $k=s+2$. Note that $k=s+2$ implies $r\neq s$, otherwise $ad-bc=0$. If $k$ is odd, then
$\frac{a}{c}$ is associated to
              $$R_1^{q_0}L_1^{q_1}\cdots R_1^{q_{s+1}} \begin{bmatrix} a'\\c' \end{bmatrix}$$
where $\frac{a'}{c'}=[q_{s+2},\ldots,q_r]$ with $\gcd(a',c')=\gcd(a,c)$, and
$\frac{b}{d}$ is associated to
$$R_1^{q_0}L_1^{q_1}\cdots L_1^{q_s} \begin{bmatrix} \gcd(b,d)\\0 \end{bmatrix}
= R_1^{q_0}L_1^{q_1}\cdots L_1^{q_s}R_1^{q_{s+1}}\begin{bmatrix} \gcd(b,d)\\0 \end{bmatrix}$$ using the fact that $R_1^m \begin{bmatrix} n\\0 \end{bmatrix}= \begin{bmatrix} n\\0 \end{bmatrix}$ for any positive integers $m$ and $n$. This implies that $R_1^{q_0}L_1^{q_1}\cdots R_1^{q_{s+1}} \begin{bmatrix} a' & \gcd(b,d)\\c' & 0\end{bmatrix}$. Again, by Theorem~\ref{equivtheorem}, it follows that
$$\dfrac{az+b}{cz+d}=\left[q_0,q_1,\ldots,q_{s-1},p,\left(\dfrac{a'z+\gcd(b,d)}{c'z}\right)^{-1}\right].$$ The proof is similar in the case where $k$ is even.
\end{proof}



\section{Complex $(u,v)$-Calkin-Wilf Trees}

So far we have considered PLFTs simply as functions and little attention has been devoted to their domain. In this section, we consider the case where $z$ is a special kind of complex number.

For any complex number $z$, let $\Re(z)$ and $\Im(z)$ represent the real and imaginary parts of $z$, respectively, and let $\mathcal{D}_0 = \{z\in\mathbb{C}:\Re(z)>0,\Im(z)>0\}$. Nathanson~\cite{N2} considers the complex Calkin-Wilf trees associated with complex roots in $\mathcal{D}_0$ using the matrices
\[ L_u:= \begin{bmatrix}
1 & 0\\
u & 1
\end{bmatrix}\text{ and }
R_v:=   \begin{bmatrix}
1 & v\\
0 & 1
\end{bmatrix},\]
where $u$ and $v$ are positive integers, to generate descendants. (Note that $L_u=L_1^u$ and  $R_v=R_1^v$.) This leads to the creation of an infinite forest of complex numbers associated with each pair $(u,v)$. As a word of caution, it is not immediately obvious that Nathanson's generalization of the Calkin-Wilf tree leads to a forest. Some justification for this fact is required (see~\cite[Theorem 2]{N2} for details).

One common property seen in various generalizations of the Calkin-Wilf tree~\cite{qCW, HMST,K,MS} is that every element appearing in a tree always has a finite number of ancestors. The goal of this article is to extend this notion to the above forest of complex numbers associated with the pair $(u,v)$. Note that the restriction to elements in $\mathcal{D}_0$ is crucial here. Without such a restriction, every element would have an infinite number of ancestors.

Given a pair $(u,v)$, if $w\in\mathcal{D}_0$ has no ancestors in its (uniquely) associated complex Calkin-Wilf tree, then we say that $w$ is a complex $(u,v)$-orphan. We begin with a characterization of the set of complex $(u,v)$-orphans due to Nathanson\footnote{The following proof of Theorem~\ref{thm1} is very similar to Nathanson's proof and was done independently by the authors after learning about the result. We include it for completeness.}~\cite{N2}.

\begin{theorem}[Nathanson, \cite{N2}]\label{thm1}
Let $\mathcal{D}_{u,v}$ be the set of complex $(u,v)$-orphans. Then $$\mathcal{D}_{u,v}=\{z\in\mathcal{D}_0:\Re(z)\leq v, |2uz-1|\geq 1\}.$$
\end{theorem}

\begin{proof}
Suppose that $z=x+iy$ is a complex $(u,v)$-orphan. If $\Re(z)> v$, then $z$ is the right child of $z-v$. This is a contradiction, so $\Re(z)\leq v$. It remains to show that $|2uz-1|\geq 1$.

Let $w=(L^u)^{-1}(z)$. A straightforward calculation shows that
\begin{align}
w &=\frac{1}{(1-ux)^2+(uy)^2}\Big(x(1-ux)-uy^2+i y\Big).\label{left}
\end{align}
In other words, $z$ is a left child unless $w\notin\mathcal{D}_0$. That is, we must have that $x(1-ux)-uy^2\leq 0$. It follows that
\begin{align*}
x(1-ux)-uy^2 & \leq 0\\
x(ux-1)+uy^2 & \geq 0\\
ux^2-x+uy^2 & \geq 0\\
x^2-\frac{1}{u}x+y^2 & \geq 0.
\end{align*}
By completing the square for $x$,
\begin{align*}
\Big(x-\frac{1}{2u}\Big)^2+y^2 & \geq \frac{1}{4u^2}.
\end{align*}
So $z$ lies on or outside of the circle centered at $\frac{1}{2u}$ of radius $\frac{1}{2u}$. In particular, $\big|z-\frac{1}{2u}\big|\geq \frac{1}{2u}$, from which the desired result follows. (See Figure~\ref{fig:cuvorph} for a graphical representation of $\mathcal{D}_{u,v}$.)
\end{proof}

\begin{figure}[ht!]
\centering
\begin{tikzpicture}[scale=1.1]
    \coordinate (y) at (0,3);
    \coordinate (x) at (4,0);
    \fill[gray!50] (0,0) -- (0,3) -- (3,3) -- (3,0) -- cycle;
    \draw[<->] (y) -- (0,0) --  (x);
    \draw (1,0) -- (1,-0.15) node [below] {$\frac{1}{2u}$};
    \draw (2,0) -- (2,-0.15) node [below] {$\frac{1}{u}$};
    \draw (3,3) -- (3,-0.15) node [below] {$v$};
    \draw (-0.15,0) node [left] {$0$} -- (0,0);
    \draw (0,-0.15) node [below] {$0$} -- (0,0);
    \draw[fill=white] (2,0) arc [radius=1, start angle=0, end angle= 180];
  \end{tikzpicture}
  \caption{The set $\mathcal{D}_{u,v}$.}
  \label{fig:cuvorph}
\end{figure}
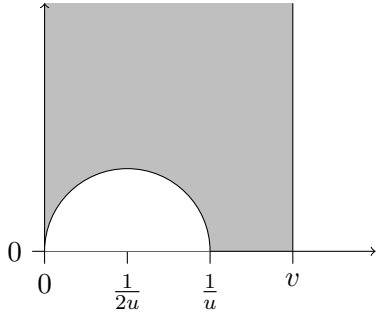

Let $\mathcal{D}_1 = \{z\in\mathcal{D}_0:|2uz-1|< 1\}$. That is, $\mathcal{D}_1$ represents the set of elements in $\mathcal{D}_0$ that are the left child of some other element in $\mathcal{D}_0$. The next result implies that there cannot be an infinite sequence of elements $\{z_n\}_{n=1}^\infty$ with $z_n\in\mathcal{D}_1$ and $(L^u)^{-1}(z_n)=z_{n+1}$ for all $n\geq 1$. Less formally, one cannot have an infinite sequence of ancestors all of which are left children.

\begin{theorem}\label{thm2}
Let $0 < y_0 \leq \frac{1}{2u}$ and $z\in\mathcal{D}_1$ be such that $\Im(z)\geq y_0$. Then $$\Im\big((L^u)^{-1}(z)\big)-\Im(z)\geq \epsilon_u(y_0)>0,$$ where $\epsilon_u(y)=\frac{2y}{1+\sqrt{1-4u^2y^2}}-y.$
\end{theorem}

\begin{proof}
As in Theorem~\ref{thm1}, suppose that $z=x+iy$ and let $w=(L^u)^{-1}(z)$. It follows from \eqref{left} that $$\Im(w)-\Im(z)=\frac{y}{(1-ux)^2+(uy)^2}-y.$$ Let $f_{u,y}(x)=\frac{y}{(1-ux)^2+(uy)^2}-y$. Then $$f'_{u,y}(x)=\frac{2uy(1-ux)}{[(1-ux)^2+(uy)^2]^2}.$$ In particular, $f'_{u,y}(x)>0$  for $x<\frac{1}{u}$, which clearly holds in this case since $|2uz-1|\leq 1$ and $y>0$. This shows that, for a fixed $y$ value, $f_{u,y}(x)$ is minimized when $x$ is as small as possible. Finding the location of the desired minimum is equivalent to determining the smaller $x$-value of the two points of intersection of the horizontal line of all complex numbers with imaginary part $y$ and the circle of radius $\frac{1}{2u}$ around $\frac{1}{2u}$. A simple computation shows that this occurs at $$x_{u,y}=\frac{1}{2u}-\sqrt{\frac{1}{4u^2}-y^2}.$$ Note that $x_{u,y}$ is a real number since we have that $0< y\leq\frac{1}{2u}$ and that $f_{u,y}(x_{u,y})=\epsilon_u(y)$.

To complete the proof, it is therefore enough to show that $\epsilon_u(y)\geq \epsilon_u(y_0)$. Differentiating $\epsilon_u(y)$ with respect to $y$, we see that $$\epsilon'_u(y)=\frac{2}{1-4u^2y^2+\sqrt{1-4u^2y^2}}-1.$$ Since $\epsilon'_u(y)>0$ for $0< y<\frac{1}{2u}$ and $\epsilon'_{u,v}(y)\to\infty$ as $y\to\frac{1}{2u}^-$, it follows that $\epsilon_u(y)$ is minimized at $y=y_0$.
\end{proof}

We now obtain the desired result.

\begin{corollary}
Every $z\in\mathcal{D}_0$ is the descendant of a complex $(u,v)$-orphan.
\end{corollary}

\begin{proof}
Suppose that there is a $z\in\mathcal{D}_0$ this is not the descendant of a $(u,v)$-orphan. That is, assume that $z$ has infinitely many ancestors $z=z_0,z_1,z_2,\dots$, all in $\mathcal{D}_0$, where either $z_{i+1}=(L^u)^{-1}(z_i)$ or $z_{i+1}=(R^v)^{-1}(z_i)$ for $i\geq 0$. If $z_i\not\in\mathcal{D}_1$ for all sufficiently large $i$, then $\lim_{i\to\infty}\Re(z_i)=-\infty$, a contradiction. So there is an infinite subsequence $i_k$, $k\geq 0$, so that  $z_{i_k}\in\mathcal{D}_1$. Using induction, it follows from Theorem~\ref{thm2} that  $\Im(z_{i_k})-\Im(z_{i_0})\geq k\epsilon_u(\Im(z_{i_0}))$. So $z_{i_k}\not\in\mathcal{D}_1$ for all sufficiently large $k$, a contradiction.
\end{proof}

\section*{Acknowledgements}

The second author was partially supported by PSC-CUNY Awards \# 67111-00 45 and \#68121-00 46, jointly funded by The Professional Staff Congress and The City University of New York.


\end{document}